\documentclass[12pt, oneside]{amsart}
\usepackage{amsmath,amsthm,amssymb}
\usepackage{geometry, xcolor}
\usepackage{hyperref}
\usepackage{tikz}
\usepackage{thmtools}
\usepackage{thm-restate}
\usepackage{cleveref}

\newcommand{\bx}{\begin{bmatrix}}
\newcommand{\ex}{\end{bmatrix}}

\newtheorem{theorem}{Theorem}
      \newtheorem{lemma}[theorem]{Lemma}
      \newtheorem{prop}[theorem]{Proposition}

\theoremstyle{definition}
      \newtheorem{observation}[theorem]{Observation}
      
      \newtheorem{definition}[theorem]{Definition}

\title{Hurwitz Equivalence of Reflection Factorizations in $G_7$}

\author{Tyler Minnick, Colin Pirillo, Sarah Racile, Yueqi Wang}

\begin{document}

\maketitle

\begin{abstract}
We prove that two reflection factorizations of a given element in an exceptional rank-$2$ complex reflection group of tetrahedral type are Hurwitz-equivalent if and only if they generate the same subgroup and have the same multiset of conjugacy classes.
\end{abstract}


\section{Introduction}

Given a group $G$, a \emph{Hurwitz move} on position $i$ of the factorization $T = (g_1, \ldots, g_n) \in G^n$ is defined as the following operation:
\[ \sigma_i(T) = (g_1, \ldots, g_{i-1},\quad g_{i+1},\quad g_{i+1}^{-1}g_i g_{i+1},\quad g_{i+2}, \ldots, g_n). \]
Two factorizations are \emph{Hurwitz equivalent} if one can be obtained by applying Hurwitz moves to the other, and the set of all factorizations Hurwitz equivalent to $T$ is called the \emph{Hurwitz orbit} of $T$ (see Section \ref{Hurwitz Action} for more on Hurwitz equivalence).

The main result in this paper is inspired by the following theorem on Hurwitz equivalence between factorizations of elements in dihedral groups.
\begin{theorem}[{Berger \cite[Theorem 2.1]{Berger}}]
\label{berger_theorem}
Let $G$ be a dihedral group and let $T$ and $V$ be length-$n$ factorizations whose elements belong to $G$. Then $T$ and $V$ are Hurwitz equivalent if and only if they factor the same element, generate the same subgroup $S$ of $G$, and have the same multiset of conjugacy classes of $S$.
\end{theorem}

In this paper we study Hurwitz equivalence in finite complex reflection groups, as suggested by Lewis-Reiner \cite{LewisReiner}. Complex reflection groups are defined as groups generated by elements called reflections (see Section \ref{crg_background}), and dihedral groups, generated by real-valued reflections, are a special case. Our main theorem shows that Berger's necessary and sufficient conditions for Hurwitz equivalence also hold in $G_7$, the largest tetrahedral rank-2 complex reflection group (defined in Section \ref{g7_bg}).

\begin{restatable}[Main Theorem]{theorem}{main}
\label{main theorem}
Let $T$ and $V$ be any two length-$n$ reflection factorizations of elements in $G_7$. Then $T$ and $V$ are Hurwitz equivalent if and only if they factor the same element, generate the same subgroup of $G_7$, and have the same multiset of conjugacy classes.
\end{restatable}

Section \ref{proof_section} contains proofs that the more complicated subgroups of $G_7$ are what we call \emph{class-searchable}, a property which we then use in Lemma \ref{main_lemma} to prove Theorem \ref{main theorem}. Section \ref{standard_forms} outlines a method to categorize reflection factorizations, creating standard factorizations for each length and class of element being factored.

\subsection*{Acknowledgments}
We would like to thank Joel Lewis for his continued mentorship throughout the duration of this project.

\section{Background}
\label{background}

\subsection{Complex reflection groups}
\label{crg_background}

We begin with some background on complex reflection groups, using the terminology and notation found in Lehrer--Taylor \cite{LehrerTaylor}.

Let $V$ be a complex vector space. Given a linear transformation $t: V\rightarrow V$, the \emph{fixed space} of $t$ is the set of elements $v \in V$ such that $t(v) = v$. We call $t$ a \emph{reflection} if its fixed space has dimension $\dim(V) - 1$. A finite group $G$ of linear transformations on $V$ is called a \emph{complex reflection group} if it is generated by its subset $R$ of reflections. Further, \emph{reflection subgroups} of $G$ are defined as subgroups of $G$ that are generated by reflections in $G$. The reflections in a complex reflection group $G$ may be partitioned into conjugacy classes of $G$, which we call \emph{reflection conjugacy classes} of the group.

Every complex reflection group is a direct product of \emph{irreducible} groups. The irreducibles were classified by Shephard and Todd \cite{ShephardTodd} as the following: every irreducible is either a member of an infinite family $G(m,p,n)$ of complex reflection groups, or one of 34 exceptional groups $G_4, \ldots, G_{37}$. Groups in the infinite family are of the form
\[
G(m, p, n) = \left\{\begin{array}{l} n \times n \text{ monomial matrices whose nonzero entries are}\\m\text{th} \text{ roots of unity with product a } \frac{m}{p}\text{th} \text{ root of unity}\end{array} \right\}
\]
for positive integers $m$, $p$, $n$ with $p \mid m$.

Let $Z(G)$ denote the center of $G$. The structure of an exceptional complex reflection group $G$ is called \emph{tetrahedral} if the quotient group $G/Z(G)$ is isomorphic to the alternating group $\text{Alt}(4)$, which represents the rotations of a tetrahedron. The group $G_7$ is the highest-order tetrahedral exceptional group, and it has reflection subgroups isomorphic to the rest of the tetrahedral exceptional groups.

\subsection{The group $G_7$} \label{g7_bg}

We now describe the structure of $G_7$ in more detail. The group is generated by three reflections $t$, $u$, and $s$, and is defined \cite[Section 5.8]{Shi} as
\[ G_7 = \left\langle s,t,u \mid s^2 = t^3 = u^3 = 1 \text{, } stu = ust = tus \right\rangle.\]
Concretely, we can take these generators to be the matrices
\[
s = \begin{bmatrix} 1 & 0 \\ 0 & -1\end{bmatrix}, \qquad 
t =  \frac{1}{4} \begin{bmatrix} (1 + \sqrt{3}) + (-1 + \sqrt{3})i & (1 + \sqrt{3}) + (-1 + \sqrt{3})i \\ (-1 + \sqrt{3}) - (1 + \sqrt{3})i &  (1 - \sqrt{3}) + (1 + \sqrt{3})i \end{bmatrix}, \qquad
u = t^\top.
\]
There are $144$ elements in $G_7$, 22 of which are reflections. These are partitioned into five reflection conjugacy classes, which we label $R_1$, $R_1^{-1}$, $R_2$, $R_2^{-1}$, and $S$.
The generators $s$, $t$, and $u$ belong to $S$, $R_1$, and $R_2$ respectively. The elements of $R_1^{-1}$ are the inverses of the elements in $R_1$, and the same is true for $R_2^{-1}$ and $R_2$. The class $S$ contains six reflections of order $2$, and the remaining reflection conjugacy classes each have four reflections of order $3$.

Each proper reflection subgroup of $G_7$ is a complex reflection group isomorphic to one of the following groups: the abelian groups $\mathbb{Z}_2$, $\mathbb{Z}_3$, $\mathbb{Z}_2\times\mathbb{Z}_2$, and $\mathbb{Z}_3\times\mathbb{Z}_3$, the dihedral group $D_{2\times 4}$, the complex reflection group $G(4, 2, 2)$, or the exceptional complex reflection groups $G_4$, $G_5$, and $G_6$. Figure \ref{classes picture} describes the reflection subgroup structure of $G_7$.

\begin{figure}\label{classes picture}
\begin{tikzpicture}[point/.style={rectangle, draw=blue!95, fill=white!1, very    thick, minimum size=5mm},
  ]
  \node (g7) at (-0.5,2) [point] {\parbox{4cm}{\centering $G_7$ \\ ($S$,         $R_1$, $R_1^{-1}$, $R_2$, $R_2^{-1}$)}};
  \node (g5) at (-5.5,0) [point] {\parbox{3.5cm}{\centering$G_5$ \\             ($R_1$, $R_1^{-1}$, $R_2$, $R_2^{-1}$)}}
      edge [-] (g7);
  \node (g4_1) at (-3.5,-2) [point] {\parbox{1.8cm}{\centering $G_4$ \\         ($R_1$, $R_1^{-1}$)}}
      edge [-] (g5);
  \node (g4_2) at (-0.5,-2) [point] {\parbox{1.8cm}{\centering $G_4$ \\           ($R_2$, $R_2^{-1}$)}}
      edge [-] (g5);
  \node (g6_1) at (-0.5,0) [point] {\parbox{2.3cm}{\centering $G_6$ \\           ($S$, $R_1$, $R_1^{-1}$)}}
      edge [-] (g4_1)
      edge [-] (g7);
  \node (g6_2) at (4, 0) [point] {\parbox{2.3cm}{\centering $G_6$ \\           ($S$, $R_2$, $R_2^{-1}$)}}
      edge [-] (g7)
      edge [-] (g4_2);
  \node (g422) at (4,-2) [point] {\parbox{3cm}{\centering $G(4, 2, 2)$         \\ ($S = S_1\cup S_2\cup S_3$)}}
      edge [-] (g6_1)
      edge [-] (g6_2);
  \node (d8_1) at (1.5,-4) [point] {\parbox{1.8cm}{\centering                   $D_{2\times4}$ \\ ($S_1,S_2$)}}
      edge [-] (g422);
  \node (d8_2) at (4,-4) [point] {\parbox{1.8cm}{\centering                   $D_{2\times4}$ \\ ($S_1,S_3$)}}
      edge [-] (g422);
  \node (d8_3) at (6.5,-4) [point] {\parbox{1.8cm}{\centering                   $D_{2\times4}$ \\ ($S_2,S_3$)}}
      edge [-] (g422);
  \node (z2z2_1) at (1.5,-6) [point] {$\mathbb{Z}_2\times\mathbb{Z}_2$}
      edge [-](d8_1)
      edge [-](d8_2);
  \node (z2z2_2) at (4,-6) [point] {$\mathbb{Z}_2\times\mathbb{Z}_2$}
      edge [-](d8_1)
      edge [-](d8_3);
  \node (z2z2_3) at (6.5,-6) [point] {$\mathbb{Z}_2\times\mathbb{Z}_2$}
      edge [-](d8_2)
      edge [-](d8_3);
  \node (z2_1) at (1.5,-8) [point] {$\mathbb{Z}_2$}
      edge [dashed] node[left] {$\times 2$} (z2z2_1);
  \node (z2_2) at (4,-8) [point] {$\mathbb{Z}_2$}
      edge [dashed] node[left] {$\times 2$} (z2z2_2);
  \node (z2_3) at (6.5,-8) [point] {$\mathbb{Z}_2$}
      edge [dashed] node[left] {$\times 2$} (z2z2_3);
  \node (z3z3) at (-5.5, -6) [point] {$\mathbb{Z}_3\times\mathbb{Z}_3$}
      edge [dashed] node[left] {$\times 4$} (g5);
  \node (z3_1) at (-3.5,-8) [point] {$\mathbb{Z}_3$}
      edge [dashed] node[left] {$\times 4$} (g4_1)
      edge [dashed] (z3z3);
  \node (z3_2) at (-0.5,-8) [point] {$\mathbb{Z}_3$}
      edge [dashed] node[left] {$\times 4$} (g4_2)
      edge [dashed] (z3z3);

\end{tikzpicture}
\caption{This picture shows the reflection subgroups of $G_7$, with each edge connecting a group to its reflection subgroups. Isomorphic copies of the same group are distinguished by the list of reflection conjugacy classes shown in parentheses. The multiplicities on the dashed edges represent how many distinct copies of the subgroup exist in the parent group.}
\end{figure}
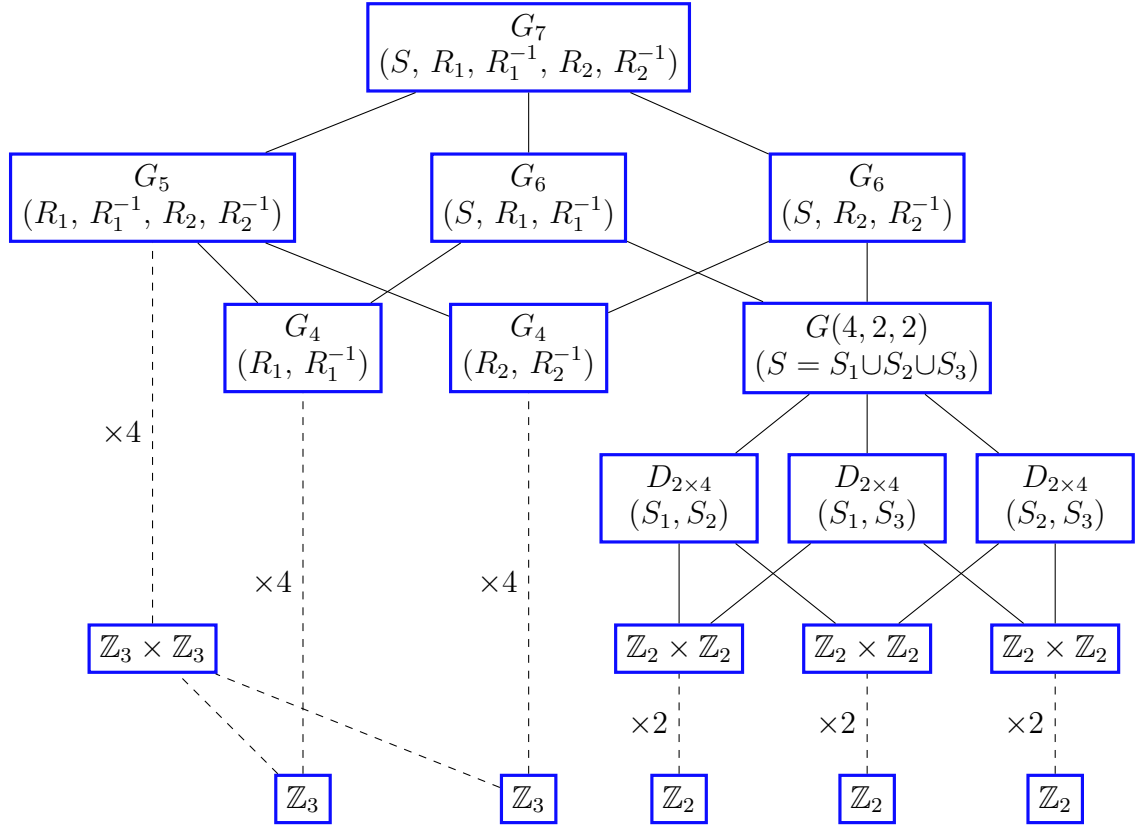

Each reflection in $S$ generates a copy of $\mathbb{Z}_2$, and each reflection in $R_1$, $R_1^{-1}$, $R_2$, or $R_2^{-1}$ generates a copy of $\mathbb{Z}_3$. This results in eight unique copies of $\mathbb{Z}_3$, as the elements in inverse classes will generate the same copies. Next, each pair of inverse elements in $R_1 \cup R_1^{-1}$ commutes with a unique pair of inverses in $R_2 \cup R_2^{-1}$, and four total copies of $\mathbb{Z}_3 \times \mathbb{Z}_3$ in $G_7$ are generated this way. There are three distinct copies of $\mathbb{Z}_2 \times \mathbb{Z}_2$ in $G_7$, generated by pairing commuting elements in $S$; and there are three copies of $D_{2\times 4}$, generated by combinations of the copies of $\mathbb{Z}_2 \times \mathbb{Z}_2$.

The groups $G_4$, $G_5$, and $G_6$ are exceptional complex reflection group with a tetrahedral structure. The group $G_4$ is generated by $\{ t, sts \}$, contains 24 elements, and has a total of 8 reflections each of order 3. These reflections are partitioned into two conjugacy classes $R_1$ and $R_1^{-1}$ (equal to the conjugacy classes of $G_7$ with the same name, mentioned above). A second copy of $G_4$ is generated by $\{u, sus\}$ and has reflection conjugacy classes $R_2$ and $R_2^{-1}$.

The group $G_5$ is generated by $\{t, u\}$ and contains 72 elements, including 16 reflections of order 3. The reflection conjugacy classes of $G_5$ are $R_1$, $R_1^{-1}$, $R_2$, and $R_2^{-1}$.

The group $G(4, 2, 2)$ is an irreducible complex reflection group generated by \\ $\{ s, tst^{-1}, t^{-1}st\}$, and is a member of the infinite family described in Section \ref{crg_background}. The elements of $G(4,2,2)$ are sixteen $2\times2$ monomial matrices whose nonzero entries are $\pm i$ and $\pm 1$, such that the nonzero entries multiply to $\pm1$. Six of these elements are reflections of order two, split evenly into three reflection conjugacy classes $S_1$, $S_2$, and $S_3$.

The group $G_6$ is generated by $\{ s, t\}$. It contains 48 elements, with 8 reflections of order 3 and 6 reflections of order 2. The reflections of order 3 are partitioned into two conjugacy classes, $R_1$, and $R_1^{-1}$, while the reflections of order 2 are all in the same conjugacy class, $S$. A second copy of $G_6$ is generated by $\{s, u\}$, with reflection conjugacy classes $S$, $R_2$, and $R_2^{-1}$.

\subsection{Hurwitz moves}
\label{Hurwitz Action}

Suppose $g$ is an element in a complex reflection group $G$. A \emph{reflection factorization} of $g$ is a tuple $T = (r_1,\ldots,r_n)$ of reflections such that $g = r_1\cdots  r_n$. The \emph{length} of the factorization is $n$. The set $\{r_1, \ldots, r_n\}$ generates some subgroup $S$ of $G$, so we say the factorization $T$ \emph{generates} $S$.
The \emph{multiset of conjugacy classes} of $T$ is the collection of conjugacy classes of $S$ whose elements appear in $T$, where the multiplicity of each class is the number of elements in $T$ belonging to that class.

Let $T=(r_1,\ldots,r_n)$ be a reflection factorization of an element $g$ in a complex reflection group. The \emph{Hurwitz move} $\sigma_i$ is an operation defined on a position $i$ in $T$, with $1 \leq i < n$, given by
$$
\sigma_i(T) = (r_1,\ldots, r_{i-1}, \quad r_{i+1}, \quad r_{i+1}^{-1} r_i r_{i+1}, \quad r_{i+2}, \ldots,r_n).
$$
Importantly, $r_{i+1}^{-1} r_i r_{i+1}$ is a reflection, and the product of the reflections in $\sigma_i(T)$ equals $g$, and so $\sigma_i(T)$ is another reflection factorization of $g$. Furthermore, $r_i$ and $r_{i+1}^{-1} r_i r_{i+1}$ belong to the same conjugacy class, so the multiset of conjugacy classes is preserved by Hurwitz moves.

The \emph{Hurwitz orbit} of a factorization $T$ is the set of all factorizations which can be produced by applying Hurwitz moves to $T$. Two factorizations belonging to the same Hurwitz orbit are said to be \emph{Hurwitz equivalent}.

\section{Class-searchable groups and proof of main theorem} \label{proof_section}

Our main theorem is restated here for convenience and with the added context of the definitions in Section \ref{background}.

\main*

It is true in general that if $T$ and $V$ are Hurwitz equivalent reflection factorizations with elements in any complex reflection group $G$, then $T$ and $V$ factor the same element, generate the same subgroup $H$ of $G$, and have the same multiset of conjugacy classes of $H$. Therefore, it is only necessary to prove that the converse holds in $G_7$. Our proof is by cases based on the subgroup of $G_7$ generated by $T$ and $V$. The dihedral case follows from Theorem \ref{berger_theorem}. The abelian case is handled by the following proposition. 

\begin{prop} \label{abelian_case}
Suppose $G$ is an abelian group, and let $T$ and $V$ be length-$n$ reflection factorizations which generate $G$, factor the same element of $G$, and have the same multiset of conjugacy classes. Then $T$ and $V$ are Hurwitz equivalent.
\end{prop}

\begin{proof}
Each element of $G$ belongs to its own conjugacy class, so the multiset of conjugacy classes of $T$ and $V$ determines exactly which elements belong to $T$ and $V$. Hurwitz moves on pairs of commuting elements simply swap the two elements within the factorization, so the elements of $T$ may be permuted by Hurwitz moves to match those of $V$.
\end{proof}

If $T$ and $V$ do not generate an abelian or dihedral group, then they generate either a tretrahedral subgroup of $G_7$, or $G(4,2,2)$. We now introduce a key component of our proof strategy for these cases.

\begin{definition}\label{class_searchable}
Let $G$ be a complex reflection group. We say that $G$ is \emph{class-searchable} at length $n$ if the following statement holds: If $T$ is any length-$n$ reflection factorization generating $G$, and if $K$ is a reflection conjugacy class of $G$ which is represented at least twice in $T$, then, for all $x \in K$, there exists a factorization $T'$ Hurwitz equivalent to $T$, such that the last factor of $T'$ is $x$ and the first $k-1$ factors generate $G$.
\end{definition}

The next five subsections prove that $G_4$, $G_5$, $G(4,2,2)$, $G_6$, and $G_7$ are class-searchable at various lengths. Finally, Lemma \ref{main_lemma} in Section \ref{final proof} uses the class-searchable property in an inductive argument. This proof uses low length base cases that may be proved simply by calculation, and the inductive step requires a group to be class-searchable at lengths greater than those calculated. If a group $G$ is class-searchable, two length $n$ factorizations $T$ and $V$ of the same element in $G$ are both equivalent to factorizations with some $x\in K$ in the right-most position. The induction is then invoked by taking two new length $n-1$ factorizations, the same as $T$ and $V$ respectively, except without the element $x$. These new factorizations are equivalent by hypothesis, and then it is easy to show that $T$ and $V$ are equivalent.

\subsection{Factorizations generating $G_4$}
\label{g4 results}

We begin with a basic fact about the Hurwitz orbits of length-2 factorizations that generate $G_4$.
\begin{prop}
\label{l2-orbits}
Suppose the reflection factorization $T = (x,y)$ generates $G_4$. If $x$ and $y$ belong to the same conjugacy class, then the Hurwitz orbit of $T$ introduces one new reflection from the same conjugacy class. If $x$ and $y$ belong to different conjugacy classes, then two new reflections are introduced in the orbit of $T$, one from each conjugacy class. 
\end{prop}
\begin{proof}
Lazreq et al.\cite{Lazreq} and our own computations in Sage.
\end{proof}

With this we show that $G_4$ is class-searchable at lengths $5$ or greater.

\begin{prop}
\label{last-place}
Let $T$ be a length $\ell \geq 5$ reflection factorization generating the subgroup $G_4$ of $G_7$. Let $m$ be the number of reflections in $T$ from conjugacy class $R_1$, and $n$ the number from $R_1^{-1}$. If $m = 0$, then, for each $x \in R_1^{-1}$, there is a factorization $V$ in the Hurwitz orbit of $T$ such that the right-most element of $V$ is $x$, and the first $\ell-1$ elements of $V$ generate $G_4$. The same statement is true when $m$ is replaced with $n$ and $R_1^{-1}$ is replaced with $R_1$. If neither $m$ nor $n$ is $0$, then the statement is true when $x$ is any reflection.
\end{prop}

\begin{proof}
The proof is by induction on $\ell$. The base case $\ell = 5$ was verified by computation. Assume the statement holds when $\ell = k - 1$, with $k \geq 6$, and let $T$ be a length-$k$ reflection factorization. To show that the proposition holds for $T$, it is enough to produce a factorization $T'$ Hurwitz equivalent to $T$, such that the last $k - 1$ elements of $T'$ generate $G_4$ and contain at least one element of each conjugacy class represented in $T$. The proposition then follows by applying the inductive hypothesis to the length-$(k-1)$ factorization obtained by ignoring the first element of $T'$.

If $m = 0$ then $n = k$ and $T = (r_1, \ldots, r_n)$, with $r_i \in R_1^{-1}$ for all $i$. If $r_2, \ldots, r_n$ generate $G_4$, then the proposition follows from the inductive hypothesis (let $T' = T$ in the previous paragraph). Otherwise $r_2, \ldots, r_n = x$ for some $x \in R_1^{-1}$, because any two distinct reflections from the same conjugacy class generate $G_4$. Thus $r_1 \in R_1^{-1} \setminus \{x\}$. One Hurwitz move sends $T$ to the factorization $T' = (x, x^{-1} \, r_1 \, x, x, \ldots, x)$. But $x^{-1} \, r_1 \, x \neq x$, so the last $k-1$ factors of $T'$ generate $G_4$, and we are done. The argument so far works if we exchange $n$ and $m$ and exchange $R_1^{-1}$ and $R_1$.

Now assume $m,n \neq 0$, and either $m$ or $n$ is $1$. If $m = 1$, then we can shift the single reflection from $R_1$ to position 2 by a series of Hurwitz moves on $T$, so that $T$ is Hurwitz equivalent to a factorization $(r_1, x, r_2, \ldots, r_n)$ with $x \in R_1$ and $r_i \in R_1^{-1}$ for all $i$. If $x, r_2, \ldots, r_n$ generate $G_4$, then the proposition follows from the inductive hypothesis. Otherwise $r_2, \ldots, r_n$ are all equal to $x^{-1}$, because any pair of noninverse reflections from different conjugacy classes generate $G_4$. Since $T$ generates $G_4$, it must be the case that $r_1 \neq x^{-1}$. By Proposition \ref{l2-orbits}, there exist $x' \in R_1$ and $y' \in R_1^{-1}$ with $x' \neq x$ such that 
\[(r_1, x, x^{-1}, \ldots, x^{-1}) \sim (y', x', x^{-1}, \ldots, x^{-1}).\]
Since $x', x^{-1}$ generate $G_4$, the proposition follows again from the inductive hypothesis.

Finally, assume $m,n \geq 2$. Using Hurwitz moves we can group the reflections in $T$ by conjugacy class, so that 
\[ T \sim (r_1, \ldots, r_m, s_1, \ldots, s_n), \]
with $r_i \in R_1$ and $s_i \in R_1^{-1}$.
If $r_2, \ldots, r_m, s_1, \ldots, s_n$ generate G4 then the statement follows by the inductive hypothesis. Otherwise, $r_2, \ldots, r_m = x$, $s_1, \ldots, s_n = x^{-1}$ for some $x \in R_1$, and $r_1 \neq x$. With one Hurwitz move we have 
\[T \sim (r_1, x, \ldots, x, x^{-1}, \ldots, x^{-1}) \sim (x, x^{-1} \, r_1 \, x, x, \ldots, x, x^{-1}, \ldots, x^{-1}).\]
Since $x^{-1} \, r_1 \, x \neq x$, the last $k-1$ reflections generate $G_4$, and we may invoke the inductive hypothesis.
\end{proof}

\subsection{Factorizations generating $G_5$}
\label{g5 results}

The exceptional complex reflection group $G_5$ is treated next. The following proposition shows that $G_5$ is class-searchable starting at length five.

\begin{prop} \label{prop:g5-last-place}
Let $T$ be a reflection factorization of length $\ell \geq 5$ generating $G_5$. If at least two reflections from some reflection conjugacy class $K \in \{R_1, R_1^{-1}, R_2, R_2^{-1}\}$ of $G_5$ appear in $T$, then for all $x \in K$, there exists a factorization $T'$, Hurwitz equivalent to $T$, of the form $(r_1, \ldots, r_{\ell-1}, x)$, such that $r_1, \ldots, r_{\ell-1}$ generate $G_5$.
\end{prop}

\begin{proof}
The proof is by induction on $\ell$, with base case $\ell=5$ done by computation. Assume the statement is true when $\ell = k-1 \geq 5$. Let $T$ be a reflection factorization of length $k$ that generates $G_5$, with conjugacy class $K$ represented at least twice in $T$. Since $T$ generates $G_5$, $T$ must contain some reflection $a$ belonging to one of $\{R_1,R_1^{-1},R_2,R_2^{-1}\} \setminus \{K, K^{-1}\}$ (otherwise $T$ would generate at most a copy of $G_4$). Move $a$ by Hurwitz moves to the first coordinate of $T$, and move two representatives of $K$ to the second and third coordinates, producing a factorization $T_1 = (a, h_1, \ldots, h_{k-1})$ with $h_1,h_2 \in K$. Let $T_1' = (h_1, \ldots, h_{k-1})$. If $T_1'$ generates $G_5$, then the proposition follows by the inductive hypothesis. Similarly if the reflections in $T_1'$ generate an isomorphic copy of $G_4$, then the proposition follows by Proposition ~\ref{last-place}.

If the subgroup of $G_5$ generated by $T_1'$ is not isomorphic to $G_4$ or $G_5$, then it is abelian, and $h_1, \ldots, h_{k-1}$ commute. But $T_1$ generates $G_5$, so $a$ and $h_1$ must not commute (noninverse commuting reflections generate $\mathbb{Z}_2 \times \mathbb{Z}_2$). By computing all Hurwitz orbits of length-2 reflection factorizations in $G_5$, we have determined that two Hurwitz moves on $(a, h_1)$ give
\[
T \sim (a, h_1, \ldots, h_{k-1}) \overset{\sigma_1^2}{\sim} (a', h_1', h_2, \ldots, h_{k-1}),
\]
where $a$ and $a'$ belong to the same conjugacy class, $h_1' \in K$, and $h_1' \neq h_1$. In particular, $h_1'$ does not commute with $h_2, \ldots, h_{k-1}$, so the subgroup $G$ of $G_5$ generated by $h_1', h_2, \ldots, h_{k-1}$ is not abelian. Thus $G$ is either $G_5$ or isomorphic to $G_4$. If $G = G_5$, the proposition follows by the inductive hypothesis; if $G \cong G_4$, use Proposition \ref{last-place}.
\end{proof}

\subsection{Factorizations generating $G(4,2,2)$}
\label{g422 results}

Proposition \ref{G422_generate} outlines the groups that are isomorphic to the reflection subgroups of $G(4, 2, 2)$ and how they are generated.

\begin{prop}
\label{G422_generate}
Taking reflections from all three conjugacy classes $S_1$, $S_2$, and $S_3$ will generate the group $G(4, 2, 2)$.
\end{prop}

\begin{proof}
The 27 combinations with at least one reflection from each class were proven by calculation through Sage to generate the entire group.
\end{proof}

The following two propositions give the necessary conditions for the right-most element of a factorization in $G(4, 2, 2)$ to be any desired element $t$.

\begin{prop}
\label{G422_conjugation}
Let $s$ and $r$ be two reflections in $G(4, 2, 2)$. If $s$ and $r$ are in the same conjugacy class, conjugating $s$ by $r$ will result in $s$. If they are not in the same class, conjugating $s$ by $r$ will result in the other reflection in the conjugacy class of $s$.
\end{prop}

\begin{proof}
When $s$ and $r$ are in the same class, they commute with each other, and so $$r^{-1}sr = r^{-1}rs = s.$$
When $s$ and $r$ are not in the same class, the result is proven by direct calculation in Sage.
\end{proof}

\begin{prop}
\label{G422_last_place_weak}
Take a factorization $T$ with length $\ell\geq3$. For any reflection $t$ in $G(4, 2, 2)$, if $T$ contains at least one reflection from each conjugacy class, then there is a factorization in the Hurwitz orbit of $T$ whose right-most element is $t$.
\end{prop}

\begin{proof}
By using Hurwitz moves, group all reflections in the same conjugacy class as $t$ on the right of the factorization. This factorization will be called $T'$, and is naturally in the Hurwitz orbit of $T$.

If $t$ is present in $T'$, it can easily be moved to the right-most position, since elements in the same class commute with each other.

Otherwise, take $r$ to be the left-most reflection in the conjugacy class of $t$, where $r$ is the $k$-th element of the factorization, and so (by construction) the $k-1$ element is not in the same class as $r$.
Use a Hurwitz move to swap the elements in the $k-1$ and $k$ position, so that $r$ is now in the $k-1$ position.
Then, use another Hurwitz move to swap the elements in the $k-1$ and $k$ position, and so the $k$-th element is now $t$ by Proposition \ref{G422_conjugation}.
Then, $t$ can be moved to the right-most position as all elements to the right of it are in the same class.
\end{proof}

With the results from Propositions \ref{G422_generate} and \ref{G422_last_place_weak}, we may give the conditions under which $G(4, 2, 2)$ is class-searchable.

\begin{prop}
\label{G422_strong_last_place}
Take a factorization $T$ with length $\ell\geq 4$, with the factorization consisting of at least one element from each conjugacy class. Let $r$ be a reflection in one of the conjugacy classes that appears multiple times. If $t$ is a reflection in the same conjugacy class as $r$, then there is a factorization whose right-most element is $t$ and whose first $\ell-1$ elements from the left generate the group $G(4, 2, 2)$ that is in the Hurwitz orbit of $T$.
\end{prop}

\begin{proof}
Use Hurwitz moves to group the reflections in $T$ by conjugacy class, with $t$ as the right-most element, which can be moved there by Proposition \ref{G422_last_place_weak}. Since $t$ is in a conjugacy class with multiple representatives in the factorization, there is still at least one reflection from each conjugacy class in the first $\ell-1$ elements of $T$. Therefore, the factorization will generate $G(4, 2, 2)$, and so the statement holds.
\end{proof}

\bigskip

\subsection{Factorizations generating $G_6$}
\label{g6 results}
The following observation denotes new notation about the conjugacy classes in $G_6$.

\begin{observation}
$G_6$'s 14 reflections are split into 3 conjugacy classes, $R_1$, $R_1^{-1}$, and $S$, where $R_1^{-1}$ contains the inverse of each element in $R_1$. Let $R' = R_1\cup R_1^{-1}$.
\end{observation}

A similar strategy as that used for the proof of the group $G(4, 2, 2)$ is used to prove that $G_6$ is class-searchable.

\begin{prop}
\label{G6_generate}
Any combination of one or more reflections from $R'$ and one or more from $S$ will generate the entire group $G_6$.
\end{prop}

\begin{proof}
Proved by calculating the 48 possible combinations of one reflection from $R'$ and one from $S$ through Sage.
\end{proof}

\begin{prop}
\label{G6_last_place}
Take a factorization $T$ with length $\ell \geq 3$. For any $t$ in the multiset of conjugacy classes of $T$, there is a factorization in the Hurwitz orbit of $T$ whose right-most element is $t$, if at least one of the following is true:
\begin{enumerate}
\item $T$ contains at least two elements from $S$ and at least one element from $R'$.
\item $T$ contains at least two elements from $R_1$ and at least one element from $S$.
\item $T$ contains at least two elements from $R_1^{-1}$ and at least one element from $S$.
\end{enumerate}
\end{prop}

\begin{proof}
For each case, the length-$3$ factorization case is proven by direct calculation in Sage. Then, for the first case, for longer lengths, use Hurwitz moves to place two elements from $S$ and one element from $R'$ into the three right-most positions of the factorization. If elements from both $R_1$ and $R_1^{-1}$ appear in the factorization, and the desired element to be found in the right-most position is an element from one of these classes, choose an element from the same class as the desired element to appear in the right-most three positions. Then, by using the same Hurwitz moves in the respective length three factorization case, any element in the multiset of $T$ may be found in the right-most position.

The same strategy holds for the second and third cases, by placing the proper elements in the three right-most positions.
\end{proof}

By combining the results of Propositions \ref{G6_generate} and \ref{G6_last_place}, the following describes the conditions under which $G_6$ is class-searchable.

\begin{prop}
\label{G6_strong_last_place}
Let $T$ be a reflection factorization with at least two elements in $R'$ and at least two elements in $S$.
Then, there is a factorization whose right-most element is any element $t$ in the multiset of $T$, and whose first $\ell - 1$ elements from the left generate the group $G_6$, that is in the Hurwitz orbit of $T$.
\end{prop}

\begin{proof}
Use Hurwitz moves to group the reflections in $T$ by conjugacy class, with $t$ as the right-most element, which can be moved there by Proposition \ref{G6_last_place}.

If $t$ is in $S$, there is still at least one reflection from $S$ in the first $\ell - 1$ elements of $T$.
Therefore, since the first $\ell - 1$ elements still includes at least one element from $R'$ and at least one element from $S$, this smaller factorization generates $G_6$, and so the statement holds.

If $t$ is in $R'$, there is still at least one reflection from $R'$ in the first $\ell - 1$ elements of $T$.
If $t$ is the only element in the factorization from $R_1$ or $R_1^{-1}$, the first $\ell-1$ elements of $T$ still includes at least one element from $R'$ and at least one element from $S$ by hypothesis, and thus in both cases will still generate $G_6$. Therefore, the statement holds.
\end{proof}

\bigskip

\subsection{Factorizations generating $G_7$}
\label{g7 results}
The following observation describes more notation about the conjugacy classes in $G_7$.

\begin{observation}
The 22 reflections in $G_7$ split into 5 conjugacy classes: $R_1$, $R_1^{-1}$, $R_2$, $R_2^{-1}$, and $S$, where $R_1^{-1}$ contains the inverses of $R_1$ and $R_2^{-1}$ contains the inverses of $R_2$. Let $R' = R_1\cup R_1^{-1}$ and $R'' = R_2 \cup R_2^{-1}$.
\end{observation}

The following propositions follow the same structure as the $G(4, 2, 2)$ and $G_6$ sections to prove the conditions where $G_7$ is class-searchable.

\begin{prop}
\label{G7_generate}
Any combination of one or more reflections from each of $R'$, $R''$, and $S$ will generate the entire group $G_7$.
\end{prop}

\begin{proof}
Proved by calculating the 384 possible combinations of reflections from $R'$, $R''$, and $S$ through Sage.
\end{proof}

\begin{prop}
\label{G7_weak}
Take a factorization $T$ with length $\ell\geq4$. For any $t$ in the multiset of $T$, if $T$ contains at least one element from $S$, $R'$, and $R''$, there is a factorization in the Hurwitz orbit of $T$ whose right-most element is $t$.
\end{prop}

\begin{proof}
At length $\ell = 4$, the proof is done by direct calculation through Sage. For longer lengths, use Hurwitz moves to place four desired elements in the four right-most positions of the factorization as follows: at least one element from $S$, $R'$, and $R''$, as well as ensuring that at least one of the four elements are from the same conjugacy class as $t$. Then, by using the same Hurwitz moves as in the respective length $\ell = 4$ case, any $t$ in the multiset of $T$ may be found in the right-most position of the factorization.
\end{proof}

The following statement describes the conditions under which $G_7$ is class-searchable.

\begin{prop}
\label{G7_strong}
Take a factorization $T$ with length $\ell\geq4$, with the factorization consisting of at least one element from $S$, $R'$, and $R''$. Let $r$ be a reflection in one of the subsets that appears multiple times. If $t$ is a reflection in the same subset as $r$, then there is a factorization whose right-most element is $t$ and whose first $\ell-1$ elements from the left generate the group $G_7$ that is in the Hurwitz orbit of $T$.
\end{prop}

\begin{proof}
Use Hurwitz moves to group the reflections in $T$ by conjugacy class, with $t$ as the right-most element, which can be moved there by Proposition \ref{G7_weak}. Since $t$ is in a subset with multiple representatives in the factorization, there is still at least one reflection from $S$, $R'$, and $R''$ in the first $\ell-1$ elements of the appropriate factorization in the Hurwitz orbit of $T$. Therefore, this factorization generates $G_7$, and so the statement holds.
\end{proof}

\subsection{Proof of main theorem}
\label{final proof}

Now that the large reflection subgroups of $G_7$ have been shown to be class-searchable, we prove our main theorem using the following lemma.

\begin{lemma} \label{main_lemma}
Let $G$ be a reflection subgroup of $G_7$, and let $n$ be the number of distinct reflection conjugacy classes in $G$. Suppose there exists an integer $m \geq n$ such that $G$ is class-searchable at lengths greater than $m$, and that Theorem \ref{main theorem} holds for factorizations of lengths $1, \ldots, m$ that generate $G$. Then Theorem \ref{main theorem} holds for factorizations of any length generating $G$.
\end{lemma}

\begin{proof}
We must show that Theorem \ref{main theorem} holds for all factorizations of length greater than $m$ generating $G$. We proceed by induction: assume Theorem \ref{main theorem} is true in $G$ at factorization lengths $1, \ldots, k$ for some $k \geq m$. Let $T$ and $V$ be length $k + 1$ reflection factorizations of some element $g \in G$, with identical multisets of conjugacy classes which generate $G$. Then the proof is complete once we show $T$ and $V$ are Hurwitz equivalent.

Since $k + 1 > n$, there exists a reflection conjugacy class $R$ of $G$ such that $R$ is represented at least twice in both $T$ and $V$. Since $G$ is class-searchable at length $k+1$, it follows that $T$ is Hurwitz equivalent to some factorization $(r_1, \ldots, r_k, x)$ and $V$ is Hurwitz equivalent to $(s_1, \ldots, s_k, x)$, where $x \in R$ and $\langle r_i \rangle = \langle s_i \rangle = G$. Now $(r_1, \ldots, r_k)$ and $(s_1, \ldots, s_k)$ both factor $gx^{-1}$, generate $G$, and share the same multiset of conjugacy classes. By assumption, Theorem \ref{main theorem} applies to length $k$ factorizations of elements in $G$, so $(r_1, \ldots, r_k) \sim (s_1, \ldots, s_k)$. Thus
\[ T \sim (r_1, \ldots, r_k, x) \sim (s_1, \ldots, s_k, x) \sim V. \qedhere\]
\end{proof}

We now have all the necessary results to complete a proof of our main theorem.

\main*

\begin{proof}
Let $G$ be the reflection subgroup of $G_7$ generated by $T$ and $V$. If $G$ is isomorphic to the dihedral group $D_{2\times 4}$, then the theorem follows from Theorem \ref{berger_theorem}. If $G$ is abelian, we may invoke Proposition \ref{abelian_case}.

For the remaining possible reflection subgroups $G$ of $G_7$, we use Lemma \ref{main_lemma}. The group $G_4$ has $2$ reflection conjugacy classes, it is class-searchable at lengths greater than $4$ by Proposition \ref{last-place}, and we have determined by computation that Theorem \ref{main theorem} holds for factorizations of length up to $4$ which generate $G_4$. Thus, by Lemma \ref{main_lemma}, the theorem holds for factorizations of any length generating $G_4$.

Similarly, Propositions \ref{prop:g5-last-place}, \ref{G422_strong_last_place}, \ref{G6_strong_last_place}, and \ref{G7_strong} respectively show that if $G$ is $G_5$, $G(4,2,2)$, $G_6$, or $G_7$, then $G$ is class-searchable at all lengths greater than some integer $m$, where $m$ is greater than the number of reflection conjugacy classes of $G$. Computations in Sage show that in each case, Theorem \ref{main theorem} holds for all factorizations that generate $G$ and have length up to $m$. Thus we find by Lemma \ref{main_lemma} that Theorem \ref{main theorem} is true for factorizations of any length which generate $G_4$, $G_5$, $G(4,2,2)$, $G_6$, or $G_7$.

This covers all possible isomorphism classes of reflection subgroups $G \leq G_7$, completing the proof.
\end{proof}

\section{Standard forms in $G_4$}
\label{standard_forms}

This section describes an approach to find a representative factorization for every Hurwitz orbit produced by a reflection factorization of an element in $G_7$. Our main theorem and proof strategy gives both the necessary and sufficient conditions for determining if two factorizations are Hurwitz equivalent. However, one downside to this approach is that it gives no information about how to categorize Hurwitz orbits. In this section, we show one method to categorize all orbits for $G_4$, which may be expanded to find all the orbits of $G_7$.

To begin, the twenty-four elements in $G_4$ can be sorted into five different categories: the identity, the negative identity, eight reflections, eight Coxeter elements (the product of two different reflections from the same class), and six other elements with determinant 1 that are the product of two reflections from different classes. When considering a factorization of one of these elements, a standard form may be constructed for any Hurwitz orbit depending on the length of the factorization and which category of element is being factored.

The following definition describes the standard forms of factorizations of the eight reflections. All congruence equations in this definition consider congruence mod 3.

\begin{definition}
\label{ref_forms_list}
Let $x$ be a reflection in the $G_4$ conjugacy class $R_1$. We will define a reflection factorization with length $l\geq8$ called $[n,m]_{G_4}^x$, the \emph{standard factorization} of $x$ with $n$ reflections in $R_1$, $m$ in $R_1^{-1}$, and whose elements generate $G_4$. By considering the determinant of the product, we can deduce that $n + 2m \equiv 1 \mod 3$. Fix some $y \in R_1^{-1}$ where $y \neq x^{-1}$. There are five unique cases to treat.

\begin{enumerate}
\item If $n \equiv 2$ (equivalently $2m \equiv 2$, $m \equiv 1$), then the standard factorization of $x$ under these conditions is defined as
$$[n,m]_{G_4}^x = (\underbrace{x, \ldots, x}_{n-2}, x, y^{-1}, y, \underbrace{y, \ldots, y}_{m - 1}).$$
\item If $n \equiv 1$, or equivalently $2m \equiv 0$, $m \equiv 0$, and $m \geq 3$, then the standard factorization is defined as
$$[n,m]_{G_4}^x = (\underbrace{x, \ldots, x}_{n}, \underbrace{y, \ldots, y}_{m}).$$
\item If $n \equiv 1$ like the previous case, but $m = 0$, then the standard form is
$$[n,m]_{G_4}^x = (\underbrace{x, \ldots, x}_{n-4}, x, y^{-1}, y^{-1}, y^{-1}).$$
\item If $n \equiv 0$, or equivalently $2m \equiv 1$, $m \equiv 2$, \emph{and} $m \geq 5$,  then the standard factorization is defined as 
$$[n,m]_{G_4}^x = (\underbrace{x, \ldots, x}_{n}, x^{-1}, x^{-1}, \underbrace{y, \ldots, y}_{m-2}).$$
\item If $n \equiv 0$ like the previous case, but $m = 2$, then the standard form is
$$[n,m]_{G_4}^x = (\underbrace{x, \ldots, x}_{n-3}, x^{-1}, x^{-1}, y^{-1}, y^{-1}, y^{-1}). $$
\end{enumerate}

\end{definition}

With the five cases defined above, the next step is to prove that any factorization of a reflection in $G_4$ is Hurwitz equivalent to one of these standard forms.

\begin{prop}
\label{ref-forms-proof}
Let $x$ be a reflection in $R_1$. Let $T$ be a reflection factorization of $x$ with length $l\geq8$ whose elements generate $G_4$, with $n$ reflections in $R_1$ and $m$ in $R_1^{-1}$. Then, Definition \ref{ref_forms_list} lists a representative of the Hurwitz orbit of $T$.
\end{prop}

\begin{proof}
To prove, we first want to show that $T$ is Hurwitz equivalent to $[n,m]_{G_4}^x$, and then that Definition \ref{ref_forms_list} describes the Hurwitz orbit of $T$. By construction, $T$ and $[n,m]_{G_4}^x$ both generate $G_4$ and have the same multiset of conjugacy classes. Since $x^3$ and $y^3$ are both equal to the identity in $G_4$, it is easy to check that each standard form is a factorization of $x$. Therefore, by Theorem \ref{main theorem}, $T$ is Hurwitz equivalent to $[n,m]_{G_4}^x$.
Next, we have to prove that Definition \ref{ref_forms_list} describes every Hurwitz orbit of a factorization of a reflection. Since the cases describe every combination of $n$ and $m$ that can construct a factorization that generates $G_4$, the list is complete.
\end{proof}

This section outlines the methodology used to classify the Hurwitz orbits for select elements. A list of the standard forms of the other elements in $G_4$ is displayed in Appendix \ref{Appendix_A}. This strategy was also considered for the other complex reflection groups discussed in this paper, although the number of cases to treat in these larger order groups proved to be prohibitively large to achieve any meaningful progress. For example, when considering just the reflections in the group $G_5$, it would be necessary to treat the separate cases for each combination of the four conjugacy classes in $G_5$, which would have significantly more representatives than those in Definition \ref{ref_forms_list}. While we believe this categorization is achievable, a full proof would require intensive amounts of time (and paper), while being unnecessary towards the understanding of our main theorem.

\newpage

\appendix
\section{Additional Standard Forms}
\label{Appendix_A}

The following definition lists the standard forms of the elements of $G_4$ not included in Section \ref{standard_forms}, that is, the sixteen non-reflection elements in $G_4$.

\begin{definition}

With order six, Coxeter elements are partitioned into two groups: $C_1 = \{c_3,c_4,c_5,c_6\}$ with determinant $-e^{\frac{5\pi i}{3}};$ $C_2 = \{c_1,c_2,c_7,c_8\}$ with determinant $e^{\frac{10 \pi i}{3}}$, where elements in $C_1$ are products of two reflections in $R_2$ and elements in $C_2$ are products of two reflections in $R_1$.
Coxeter elements:

     In a factorizations of $c \in C_1$, then $n + 2m \equiv 1 \pmod{3}$. 

    \begin{itemize}
    \item If $n \equiv 0 \pmod{3}, m \equiv 2 \pmod{3}:$ 

    Standard form: $(\underbrace{y_1^{-1},\cdots, y_1^{-1}}_\text{n}, y_1,y_2, \underbrace{y_2, \cdots, y_2}_\text{m-2})$ where $y_1 \cdot y_2 = c$ and $y_1,y_2 \in R_2$.

    \item If $n \equiv 1 \pmod{3}, m \equiv 0 \pmod{3}:$

    Standard form for $m \geq 3$: $(\underbrace{y_1^{-1},\cdots, y_1^{-1}}_\text{n-1},y_1^{-1}, y_1,y_1,y_2, \underbrace{y_2, \cdots, y_2}_\text{m-3})$ where $y_1 \cdot y_2 = c$ and $y_1,y_2 \in R_2$.

    Standard form for $m = 0$: $(\underbrace{y_1^{-1},\cdots, y_1^{-1}}_\text{n-4},y_1^{-1}, y_1^{-1},y_2^{-1},y_2^{-1})$ where $y_1\cdot y_2 = c$ and $y_1,y_2 \in R_2$.

    \item If $n \equiv 2 \pmod{3}, m \equiv 1 \pmod{3}:$

    Standard form: $(\underbrace{y_1^{-1}, \cdots, y_1^{-1}}_\text{n-2}, y_1^{-1}, y_1^{-1},y_2, \underbrace{y_2, \cdots, y_2}_\text{m-1})$ where $y_1\cdot y_2 = c$ and $y_1,y_2 \in R_2$.  

\end{itemize}

And for the factorizations of $c \in C_2$, $n + 2m \equiv 2 \pmod{3}$.

  \begin{itemize}
    \item If $n \equiv 0 \pmod{3}, m \equiv 1 \pmod{3}:$

    Standard form: for $m \geq 4$: $(\underbrace{x_1, \cdots, x_1}_\text{n}, x_1^{-1}, x_1^{-1}, x_2^{-1}, x_2^{-1}, \underbrace{x_2^{-1}, \cdots, x_2^{-1}}_\text{m-4})$ with $x_1,x_2 \in C_1$, $x_1 \cdot x_2 = c$.

    Standard form: for $m = 1$: $(\underbrace{x_1, \cdots, x_1}_\text{n-3}, x_1, x_2, x_2, x_2^{-1})$ with $x_1,x_2 \in C_1$, $x_1 \cdot x_2 = c$.

    \item If $n \equiv 1 \pmod{3}, m \equiv 2 \pmod{3}:$

    Standard form: $(\underbrace{x_1, \cdots, x_1}_\text{n-1}, x_1, x_2^{-1},x_2^{-1}, \underbrace{x_2^{-1}, \cdots, x_2^{-1}}_\text{m-2})$.

    \item If $n \equiv 2 \pmod{3}, m \equiv 0 \pmod{3}:$

    Standard form: $(\underbrace {x_1, \cdots, x_1}_\text{n-2}, x_1, x_2, \underbrace{x_2^{-1}, \cdots, x_2^{-1}}_\text{m})$.

  \end{itemize}

The next element to check is the identity element, $o_1$, which naturally has determinant 1. It has the following standard forms.

\begin{itemize}
    \item If $n \equiv 0 \pmod{3}, m \equiv 0 \pmod{3}:$

    Standard form: $(\underbrace{x,\cdots,x}_\text{n}, \underbrace{y,\cdots,y}_\text{m})$ with $x\in {R_1}, y\in {R_2}$.

    \item If $n \equiv 1 \pmod{3}, m \equiv 1 \pmod{3}:$

    Standard form: $(\underbrace{x,\cdots,x}_\text{n-1},x, x^{-1} \underbrace{y,\cdots,y}_\text{m-1})$ with $x\in {R_1}, y\in {R_2}$.

    \item If $n \equiv 2 \pmod{3}, m \equiv 2 \pmod{3}:$

    Standard form: $(\underbrace{x,\cdots,x}_\text{n-2},x_2, x_1, x_1^{-1}, x_2^{-1} \underbrace{y,\cdots,y}_\text{m-2})$ with $x,  x_1, x_2 \in {R_1}, y\in {R_2}$.

    \end{itemize}

The next element is the negative identity. This is only one element, $o_7$, of order two with determinant $1$. So we have $n + 2m \equiv 0 \pmod{3}$. Note that $o_7 = x_1\cdot x_2\cdot x_3$ where $x_1,x_2,x_3 \in {R_1}$. Also, since $o_7$ has order $2$, $x_1\cdot x_2\cdot x_3 = x_3^{-1}\cdot x_2^{-1}\cdot x_1^{-1}$. For $n+m = l \geq 4:$

    \begin{itemize}
    \item If $n \equiv 0 \pmod{3}, m \equiv 0 \pmod{3}:$

    Standard form for $n \geq 3$: $(\underbrace{x_1, \cdots, x_1}_\text{n-3}, x_1,x_2,x_3, \underbrace{x_3^{-1}, \cdots, x_3^{-1}}_\text{m})$.
    Or:

    Standard form for $n = 0$: $(x_3^{-1},x_2^{-1},x_1^{-1}, \underbrace{x_1^{-1}, \cdots, x_1^{-1}}_\text{m-3})$.

    \item If $n \equiv 1 \pmod{3}, m \equiv 1 \pmod{3}:$

    Standard form for $n = 1$: $(x_1, x_2^{-1}, x_2^{-1}, x_3^{-1}, x_3^{-1}, \underbrace{x_3^{-1}, \cdots, x_3^{-1}}_\text{m-4})$.

    Standard form for $n \geq 4$: $(\underbrace{x_1,\cdots, x_1}_\text{n-4}, x_1, x_2,x_3, x_3, x_3^{-1},\underbrace{x_3^{-1},\cdots, x_3^{-1}}_\text{m-1})$.

    \item If $n \equiv 2 \pmod{3}, m \equiv 2 \pmod{3}:$

    Standard form: $(\underbrace{x_1,\cdots, x_1}_\text{n-2}, x_1, x_2,x_3^{-1}, x_3^{-1},\underbrace{x_3^{-1},\cdots, x_3^{-1}}_\text{m-2})$.

    \end{itemize}

The final elements to check consist of the other six elements with determinant $1$. We call these elements $\{o_2, o_3, o_4, o_5,o_6,o_8\}$. So $n + 2m \equiv 0 \pmod{3}$. For $n+m = l \geq 3:$

    \begin{itemize}
    \item If $n \equiv 0 \pmod{3}, m \equiv 0 \pmod{3}:$

    Standard form for $m \geq 3$: $(\underbrace{x_1, \cdots, x_1}_\text{n}, x_1^{-1}, x_1^{-1}, y_1, \underbrace{y_1, \cdots, y_1}_\text{m-3})$ where $x_1 \in {R_1}, y_1 \in {R_2}, x_1\cdot y_1 = o$ for some $o$ in the set.

    Standard form for $m = 0$: $(\underbrace{x_1, \cdots, x_1}_{n-3}, x_1, y_1^{-1}, y_1^{-1})$ where $x_1 \in {R_1}, y_1 \in {R_2}, x_1\cdot y_1 = o$ for some $o$ in the set.

    \item If $n \equiv 1 \pmod{3}, m \equiv 1 \pmod{3}:$

    Standard form : $(\underbrace{x_1, \cdots, x_1}_{n-1}, x_1, y_1, \underbrace{y_1, \cdots, y_1}_{m-1})$ where $x_1 \in {R_1}, y_1 \in {R_2}, x_1\cdot y_1 = o$ for some $o$ in the set.

    \item If $n \equiv 2 \pmod{3}, m \equiv 2 \pmod{3}:$

    Standard form: $(\underbrace{x_1, \cdots, x_1}_\text{n-2}, x_1^{-1}, x_1^{-1},y_1^{-1}, y_1^{-1}, \underbrace{y_1, \cdots, y_1}_\text{m-2})$ where $x_1 \in {R_1}, y_1 \in {R_2}, x_1\cdot y_1 = o$ for some $o$ in the set.

\end{itemize}
\end{definition}

In conjunction with Section \ref{standard_forms}, every possible Hurwitz orbit of a reflection factorization of an element in $G_4$ has been listed.

\end{document}